\documentclass{article}

\usepackage[
	a4paper,
	margin=25mm
]{geometry}
\usepackage{setspace}
\usepackage{sectsty} 
\usepackage{titlesec}

\usepackage{amsmath, amssymb, amsthm}
\usepackage{mathtools}

\usepackage[
	setpagesize=false,
	colorlinks=true,
	linkcolor=blue,
        citecolor=blue,
        urlcolor=black,
	pdfencoding=auto,
	psdextra,
]{hyperref}
\usepackage{cleveref}

\usepackage{etoolbox} 
\usepackage[shortlabels]{enumitem}
\usepackage{xparse} 

\usepackage{graphicx}
\usepackage[dvipsnames]{xcolor}
\usepackage{tikz}

\usepackage{todonotes}
\usepackage[
    deletedmarkup=sout,
    commentmarkup=uwave,
    authormarkuptext=name
]{changes}

\usepackage{comment}

\setlist[enumerate,1]{label=(\arabic*), ref=(\arabic*)}
\setlist[enumerate,3]{label=(\roman*), ref=(\roman*)}

\allsectionsfont{\boldmath} 

\titlelabel{\thetitle.\quad}

\theoremstyle{plain}
\newtheorem{theorem}{Theorem}[section]
\newtheorem{lemma}[theorem]{Lemma}
\newtheorem{corollary}[theorem]{Corollary}
\newtheorem{proposition}[theorem]{Proposition}

\newtheorem{question}[theorem]{Question}

\newtheorem{claim}{Claim}[theorem]
\newtheorem*{claim*}{Claim}

\makeatletter
\newenvironment{claimproof}[1][Proof]{\par
	\pushQED{\qed}%
	
	\normalfont \topsep6\p@\@plus6\p@\relax
	\trivlist
	\item[\hskip\labelsep
	\textit{#1}\@addpunct{.}~]\ignorespaces
}{%
	\popQED\endtrivlist\@endpefalse
}
\makeatother

\newlist{Cases}{enumerate}{3}
\setlist[Cases]{parsep=0pt plus 1pt}
\setlist[Cases,1]{wide=0pt, listparindent=\parindent,
    label = \textbf{Case~\arabic*:}, ref = \arabic*}
\setlist[Cases,2]{wide=\parindent, listparindent=\parindent,
    label = \textbf{Case~\arabic{Casesi}-\arabic{Casesii}:}}

\crefname{Casesi}{case}{cases}

\newcounter{case}
\AtBeginEnvironment{proof}{\setcounter{case}{0}}

\crefname{case}{case}{cases}

\theoremstyle{definition}
\newtheorem{definition}[theorem]{Definition}

\newcommand{\calH}{\mathcal{H}}

\newcommand{\ve}{\varepsilon}

\makeatletter
\NewDocumentCommand{\xsideset}{mmme{_^}}{%
  \mathop{%
    \settowidth{\dimen0}{$\m@th\displaystyle#3$}%
    \dimen0=.5\dimen0
    \settowidth{\dimen2}{$%
      \m@th\displaystyle#3%
      \IfValueT{#4}{_{#4}}%
      \IfValueT{#5}{^{#5}}%
    $}%
    \dimen2=.5\dimen2
    \advance\dimen2 -\dimen0
    \sbox6{\scriptspace\z@$\displaystyle{\vphantom{#3}}#1$}
    \sbox8{\scriptspace\z@$\displaystyle{\vphantom{#3}}#2$}
    \ifdim\wd6>\dimen2 \kern\dimexpr\wd6-\dimen2\relax\fi
    {%
     \mathop{\llap{\copy6}{\displaystyle#3}\rlap{\copy8}}\limits
     \IfValueT{#4}{_{#4}}%
     \IfValueT{#5}{^{#5}}%
    }%
    \ifdim\wd8>\dimen2 \kern\dimexpr\wd8-\dimen2\relax\fi
  }%
}
\makeatother

\let\originalleft\left
\let\originalright\right
\renewcommand{\left}{\mathopen{}\mathclose\bgroup\originalleft}
\renewcommand{\right}{\aftergroup\egroup\originalright}

\makeatletter
\newcases{lrcases*}
  {\quad}
  {$\m@th{##}$\hfil}
  {{##}\hfil}
  {\lbrace}
  {\rbrace}
\makeatother

\usetikzlibrary{matrix,arrows,decorations.pathmorphing}
\usetikzlibrary{positioning}
\usetikzlibrary{graphs} 
\usetikzlibrary{graphs.standard}

\title{Dirac's theorem for linear hypergraphs}
\author{Seonghyuk Im\thanks{Department of Mathematical Sciences, KAIST, South Korea and Extremal Combinatorics and Probability Group (ECOPRO), Institute for Basic Science (IBS). Email:{\tt \{seonghyuk, hyunwoo.lee\}@kaist.ac.kr}}
 \and Hyunwoo Lee\footnotemark[1]}
\date{\today}

\begin{document}
\maketitle

\begin{abstract}
    Dirac's theorem states that any $n$-vertex graph $G$ with even integer $n$ satisfying $\delta(G) \geq n/2$ contains a perfect matching.
    We generalize this to $k$-uniform linear hypergraphs by proving the following. 
    Any $n$-vertex $k$-uniform linear hypergraph $H$ with minimum degree at least $\frac{n}{k} + \Omega(1)$ contains a matching that covers at least $(1-o(1))n$ vertices.
    This minimum degree condition is asymptotically tight and obtaining a perfect matching is impossible with any degree condition.
    Furthermore, we show that if $\delta(H) \geq (\frac{1}{k}+o(1))n$, then $H$ contains almost spanning linear cycles, almost spanning hypertrees with $o(n)$ leaves, and ``long subdivisions'' of any $o(\sqrt{n})$-vertex graphs.
\end{abstract}

\section{Introduction}\label{sec:intro}
Embedding spanning structures in a dense (hyper-)graph is one of the central topics of extremal graph theory. 
A fundamental result by Dirac~\cite{dirac1952some} states that every $n$-vertex graph $G$ with minimum degree at least $n/2$ contains a Hamiltonian cycle, hence a perfect matching if $n$ is even. 
Inspired by Dirac's theorem, the minimum degree conditions for the existence of spanning structures of the graph have been extensively studied.
For example, perfect $K_r$-tiling by Hajnal and Szemer\'{e}di~\cite{hajnal1970proof}, spanning trees and powers of Hamiltonian cycles by Koml\'{o}s, S\'{a}rk\"{o}zy, and Szemer\'{e}di~\cite{komlos1995proof, komlos1998posa,komlos1998proof, komlos2001spanning}, and the bandwidth theorem by B\"{o}ttcher, Schacht and Taraz~\cite{bottcher2009proof}.

It is natural to consider such statements for hypergraphs.
A hypergraph $H$ is called a \emph{$k$-uniform hypergraph} or simply a \emph{$k$-graph} if every edge has exactly $k$ vertices.
Unless we mention otherwise, all the $k$-graphs are simple in this paper.
R\"{o}dl, Ruci\'{n}ski, and Szemere\'{e}di~\cite{rodl2009perfect} proved that for all sufficiently large $n$ that is divisible by $k$, any $n$-vertex $k$-graph hypergraph with minimum $(k-1)$-codegree at least $(1 + o(1))\frac{n}{2}$ contains a perfect matching as well as a tight Hamiltonian cycle. 
Starting from this result, minimum (co-)degree conditions for the existence of various spanning sparse structures in the given hypergraphs were considered. (See excellent surveys in~\cite{kuhn2009embedding,kuhn2014hamilton, zhao2016recent} for such examples.)

Among many spanning structures considered, linear hypergraphs form an interesting class of sparse hypergraphs.
A hypergraph is called \emph{linear} if every pair of edges share at most one vertex.
Linear hypergraphs form a natural family of sparse hypergraphs that includes many interesting hypergraphs; such as perfect matchings, loose cycles, and hypertrees. 
Sufficient conditions for a hypergraph $H$ to contain a specific linear hypergraph have been at the core of attention.
For instance, the minimum (co-)degree condition for the existence of a loose Hamilton cycle was studied by many different groups~\cite{Bus2013, Czygrinow2014, Han2015, Han2010, Keevash2011, Kuhn2006loose} and spanning hypertrees in $3$-graph was considered by Pehova and Petrova~\cite{pehova2023minimum}. (See Section~\ref{subsec:hypertree} for the definition of hypertrees.)

While all these results state that a dense hypergraph $H$ contains a specific spanning linear substructure, it is also natural to consider a sparse hypergraph $H$.
For example, if $H$ itself is a linear hypergraph, then what condition of $H$ guarantees a desired linear hypergraph as a subgraph?
Hence, the following natural question arises.

\begin{question}\label{ques:linear-dirac}
    Let $k \geq 2$ be a positive integer. 
    Let $H$ and $F$ be linear $k$-graphs. 
    What is the minimum degree condition of $H$ that guarantees containment of $F$?
\end{question}

On the other hand, finding a spanning copy of $F$ in a linear hypergraph $H$ is not possible in general even if $F$ is a perfect matching, which is the most simple spanning structure possible.
Even if $H$ is a Steiner triple system, a linear $3$-graph with maximum possible minimum degree, it may not contain a perfect matching with $\lfloor \frac{n}{3} \rfloor$ hyperedges.
Deciding the size of the largest matching guaranteed in any Steiner triple system was a long-standing open problem of Brouwer and was confirmed to be $\lfloor \frac{n-4}{3} \rfloor$ by Montgomery~\cite{montgomery2023proof} very recently.
It is still widely open for higher uniformity.
As finding spanning structures is not possible in general, one can instead consider near-spanning structures.
Dellamonica and R\"{o}dl~\cite{Dellamonica2019} proved that a Steiner triple system contains a loose path on $(1-o(1))n$ vertices.
Elliott and R\"{o}dl~\cite{Elliott2019} proved that any Steiner triple system contains all subdivision trees on $(1-o(1))n$ vertices, which are some special type of hypertrees.
They further conjectured that every hypertree of $(1-o(1))n$ vertices can be found in any Steiner triple system.
This conjecture was proved by the first author, Kim, Lee, and Methuku~\cite{Seonghyuk2022}.
These results show that embedding a near-spanning linear hypergraph into a linear hypergraph $H$ is highly non-trivial even if $H$ is a Steiner triple system, which is a ``complete'' linear $3$-graph.
This makes the search for the minimum degree condition in Question~\ref{ques:linear-dirac} very interesting.

On the other hand, there are natural restrictions for the hypergraphs $F$ in Question~\ref{ques:linear-dirac}.
The recent breakthrough by Kwan, Sah, Sawhney, and Simkim~\cite{kwan2022highgirth} as well as the asymptotic results in~\cite{Bohman2019, Glock2020} show that there exists (almost-)Steiner triple system with arbitrarily large girth.
Even more, \cite{delcourt2022finding, Glock2022conflictfree} proved that there exists ``almost complete'' linear $k$-graphs with large girth.
These results show that one can only consider linear hypergraph $F$ with large girth in Question~\ref{ques:linear-dirac}, so it makes sense for us to focus on the case when $F$ is locally tree-like, such as matchings, hypertrees, loose cycles, and long ``subdivisions''.


\subsection{Main theorems}

Our first main result is the following theorem on the Dirac-type condition for the existence of almost perfect matchings in linear hypergraphs.

\begin{theorem}\label{thm:min_deg_almost_matching}
    For $\ve > 0$ and a positive integer $k\geq 2$, there exist positive integers $C = C(\ve, k)$ and $n_0 = n_0(\ve, k)$ such that the following holds for all $n \geq n_0.$
    If $H$ is an $n$-vertex linear $k$-graph with $\delta(H) \geq \frac{n}{k} + C$, then $H$ contains a matching $M$ that covers at least $(1 - \ve)n$ vertices.
\end{theorem}

We note that the case $k = 2$ follows from Dirac's theorem. 
Indeed, by applying the following proposition with $m=\varepsilon n$, one can see that the minimum degree condition in \Cref{thm:min_deg_almost_matching} is asymptotically best possible for all $k \geq 2$. 

\begin{proposition}\label{prop:lower-bound}
    Let $n, m, k$ be positive integers with $m < \frac{n}{k}$ and $n$ is sufficiently large.
    Then there exist infinitely many $n$-vertex linear $k$-graph $H$ such that the following holds.
    \begin{itemize}
        \item[$(1)$] $\delta(H) > \frac{n}{k} - m -2$,
        \item[$(2)$] there is no matching covers more than $n - km$ vertices. 
    \end{itemize}
\end{proposition}
\begin{proof}
    Let $n'$ be the largest integer such that $n \leq kn' - \frac{km}{k-1}$. 
    By~\cite{Chowla1960}, there is a family of $k-2$ pairwise orthogonal Latin squares of order $n'$, thus the complete $k$-partite graph $K_{n', n', \ldots, n'}$ can be decomposed into copies of $K_k$'s.
    By considering each clique in the decomposition as hyperedges, one obtains ``complete'' $k$-partite linear $k$-graph.
    Let $H$ be the linear $k$-graph obtained by deleting $\lceil\frac{km}{k-1}\rceil$ vertices of one single part from this complete $k$-partite linear $k$-graph.
    Then $H$ is an $n$-vertex linear $k$-graph with $\delta(H) = n' - \lceil\frac{km}{k-1}\rceil \geq n' - \frac{km}{k-1} -1 > \frac{n}{k} - m-2$ and $H$ has no matching covers more than $k(n'-\lceil\frac{km}{k-1}\rceil) \leq n - km$ vertices.
\end{proof}
We note that this example is obtained from a balanced ``complete'' $k$-partite linear $k$-graph by deleting vertices from one part.

In fact, we prove the stronger statement regarding almost spanning $T$-tilings for hypertrees $T$ of bounded size. Again, see \Cref{subsec:hypertree} for the definition of hypertrees.

\begin{theorem}\label{thm:tree-tiling}
    For given $\ve > 0$, a positive integer $k \geq 2$ and a linear $k$-hypertree $T$, there exist positive real numbers $C = C(\ve, T)$ and $n_0 = n_0(\ve, T)$ such that the following holds for all $n \geq n_0.$
    If $H$ is an $n$-vertex linear $k$-graph with $\delta(H) \geq \frac{n}{k} + C$, then $H$ contains a collection of vertex-disjoint copies of $T$ that covers at least $(1 - \ve)n$ vertices.
\end{theorem}
We note that \Cref{thm:min_deg_almost_matching} is a corollary of \Cref{thm:tree-tiling}.

By combining \Cref{thm:tree-tiling} and the standard reservoir technique, we show that one can embed almost spanning ``path-like'' structures including linear cycles. 
To state this result, we need the following definitions.

We say that an edge $e$ of a linear hypertree $T$ is a \emph{leaf-edge} if all but at most one vertex in $e$ has degree one.
For a $(2-)$graph $G$, a \emph{$k$-uniform subdivision} of $G$ is a linear $k$-graph $F$ obtained by the following.
For each vertex of $G$, there is a corresponding vertex in $F$, called the \emph{center} vertices.
For each edge $uv$ of $G$, we add a linear path connecting the vertices corresponding to $u$ and $v$ in $F$ where internal vertices are disjoint from the other part of the hypergraph.
The resulting graph $F$ is called a $k$-uniform subdivision of $G$.
For instance, a linear cycle is a $k$-uniform subdivision of $K_3$.

\begin{theorem}\label{thm:min_deg_almost_subdivision}
    For $\ve > 0$ and a positive integer $k \geq 2$, there exist positive real numbers $\mu = \mu(\ve, k)$, and $n_0 = n_0(\ve, k)$ such that the following holds for all $n \geq n_0.$ 
    Let $H$ be an $n$-vertex $k$-graph with $\delta(H) \geq \frac{n}{k} + \ve n$ and $F$ is a $k$-graph on at most $(1 - \ve)n$ vertices which is either
    \begin{itemize}
        \item a linear $k$-hypertree with at most $\mu n$ leaf edges or 
        \item a $k$-uniform subdivision of a graph $G$ with at most $\mu \sqrt{n}$ vertices where each edge of $G$ is replaced by a path of length at least two.
    \end{itemize}
    Then $H$ contains $F.$
\end{theorem}
We note that by the previous observation, \Cref{thm:min_deg_almost_subdivision} provides the existence of an almost spanning linear cycle.

In fact, considering only the hypergraphs with not too many leaves is necessary as the statement is not true in general for the hypertrees with many leaves.
For example, if $F$ is an almost spanning star, which is a hypertree with a vertex belonging to all edges, then we need minimum degree of $H$ to be $(\frac{1}{k-1}-o(1))n$ to be able to find $F$ in $H$.
Similarly, it is not difficult to construct more almost spanning hypertrees with many leaves and small diameters that do not embed into the aforementioned ``complete'' $k$-partite linear $k$-graph mentioned in the proof of \Cref{prop:lower-bound}.

In order to prove \Cref{thm:tree-tiling}, we first determine the minimum degree threshold for the existence of a perfect fractional matching in Section~\ref{sec:regularization}, which is an LP relaxation of perfect matching, in a linear hypergraph.
Then we show that using a regularization lemma (\Cref{lem:make-pseudorandom}) and applying the Pippenger-Spencer theorem (\Cref{thm:pippenger_spencer}) on an appropriate auxiliary graph, we find an almost $T$-tiling in Section~\ref{sec:tree-tiling}.
To prove \Cref{thm:min_deg_almost_subdivision}, we additionally use a lemma on the structure of hypertree (\Cref{lem:leaf_path_hypertree}) and connection technique in Section~\ref{sec:subdivison}.
Concluding remarks are given in Section~\ref{sec:concluding_remarks}.


\section{Prelimilaries}

\subsection{General notations}
For a $k$-graph $H$ and a set $S \subseteq V(H)$ let $d_H(S)$ be the \emph{codegree} of the set $S$ in $H$, which is the number of edges containing $S$.
When $S = \{v\}$, we simply denote $d_H(\{v\})$ by $d_H(v)$ and call it be the \emph{degree} of the vertex $v$.
We denote by $\Delta(H)$ and $\delta(H)$ as the maximum degree and minimum degree of $H$, respectively.
Similarly, let $\Delta_2(H)$ be the maximum ($2-$)codegree is defined by maximum codegree $d_H(S)$ among all size $2$ sets $S \subseteq V(H)$.
For a vertex set $X \subseteq V(H)$ and a vertex $v \in V(H)$, we denote by $d_H(v; X)$ the number of edges that are incident to $v$ with all other vertices in $X$. 
For a vertex $v \in V(H)$, let $N_H(v)$ be the set of vertices $u$ such that $\{u, v\}$ is contained in an edge of $H$, called the \emph{neighborhood} of $v$.
When $H$ is a multi-hypergraph, we denote by $H_{simp}$ the simplification of $H$, that is to replace all parallel edges with one edge. 
For an edge $e\in E(H)$, we write $\mu_{H}(e)$ for the multiplicity of the edge $e$ in $H$.

We use $\ll$ to denote the hierarchy between constants. 
If we state that a statement holds when $0<a \ll b, c \ll d$, it means there exist functions $g_1, g_2:(0, 1] \rightarrow (0, 1]$ and $f:(0, 1]^2 \rightarrow (0, 1]$ such that the statement holds whenever $0<b \leq g_1(d)$, $0<c \leq g_2(d)$, and $0<a \leq f(b, c)$.
We do not explicitly determine these functions.

\subsection{Matchings in pseudorandom hypergraphs}
As conditions on degrees and codegrees will frequently appear in later sections, we define the following terminologies to simplify the notation.
A multi-hypergraph $H$ is called \emph{$(D, \tau, \delta)$-pseudorandom} if for all $v\in V(H)$, we have $(1 - \tau)D \leq d_H(v) \leq (1 + \tau)D$ and $\Delta_2(H) \leq \delta D$.

\begin{theorem}[Pippenger-Spencer \cite{Pippenger89}]\label{thm:pippenger_spencer}
    For every $k \geq 2$ and $\delta>0$, there exist $\delta_0 = \delta_0(\delta, k)>0$ and $n_0 = n_0(\delta, k)$ such that if $H$ is a multi-$k$-graph with at least $n_0$ vertices and $(D, \delta_0, \delta_0)$-pseudorandom, then $H$ can be properly edge-colored with $(1+\delta)D$ many colors. 
    In particular, $H$ has a matching that covers all but at most $2k\delta |V(H)|$ vertices.
\end{theorem}

We now define a fractional matching which is a solution of linear programming relaxation of a matching.
\begin{definition}
    A \emph{fractional matching} of a $k$-graph $H$ is a function $\psi:E(H) \rightarrow [0, 1]$ such that for every vertex $v \in V(H)$, the inequality $\sum_{v\in e\in E(H)} \psi(e) \leq 1$ holds.
    If the equality holds for every $v \in V(H)$, then say a fractional matching is \emph{perfect}.
\end{definition}

We say a perfect fractional matching $\psi$ on hypergraph $H$ is \emph{$\delta$-pseudorandom} for some real number $\delta > 0$ if every pair of vertices $\{u, v\}\in \binom{V(H)}{2}$ satisfies the inequality $\sum_{\{u, v\}\subseteq e\in E(H)} \psi(e) \leq \delta.$
It is well-known that an optimal solution to linear programming problems can be attained in the rational. Hence, we may assume that all the values of $ \psi $ are rational.
The next lemma by the second author~\cite{lee2023towards} shows that a hypergraph having many perfect fractional matchings also contains a pseudorandom fractional matching.

\begin{lemma}[Lee~\cite{lee2023towards}]\label{lem:make-pseudorandom}
    Let $C, D > 0$ be positive integers and $\gamma, \delta, \eta > 0$ be positive real numbers that satisfy the inequality $\delta D (\gamma + 2^{-\gamma C}) \leq \frac{\eta}{2}.$
    Let $H$ be a multi-hypergraph with $\delta(\calH) \geq D$ and $\Delta_2(\calH) \leq \delta D.$ Assume for all sub-hypergraph $F$ of $H$ with $\Delta(F) \leq C$, we can find a perfect fractional matching in $H\setminus F.$ Then $H$ has an $\eta$-pseudorandom fractional matching.
\end{lemma}


\subsection{Structure of hypertrees}\label{subsec:hypertree}
In this section, we introduce several notations on hypertrees and several lemmas on their structure.
\begin{definition}
    For an integer $\ell \geq 2$, 
    a linear $k$-hypergraph $P$ with a pair of vertex set $$V(P)=\{x_1, x_2, \ldots, x_{\ell+1}\} \cup \{y_{i, j} \mid 1 \leq i \leq \ell, 1 \leq j \leq k-2 \}$$ 
    and the edge set $$E(P) = \bigl \{ \{x_1, y_{1, 1}, \ldots, y_{1, k-2}, x_2\}, \{x_2, y_{2, 1}, \ldots, y_{2, k-2}, x_3\}, \ldots, \{x_\ell, y_{\ell, 1}, \ldots, y_{\ell, k-2}, x_{\ell + 1}\} \bigl \}$$     
    is called a \emph{($k$-)linear path} of length $\ell$.
\end{definition}

For the convention, we also say a $k$-graph with $k$ vertices and only one edge as a linear path of length one. 
We say two sets $\{x_1, y_{1, 1}, \ldots, y_{1, k-2}\}$ and $\{x_{\ell+1}, y_{\ell, 1}, \ldots, y_{\ell, k-2}\}$ are \emph{ends} of the linear path $P$.
If $P$ is a length $\ell \geq 2$ linear path with its ends $X, Y$, we say a path $P$ together with two vertices $u \in X$ and $v \in Y$ is called a \emph{linear $u$--$v$ path} and $u, v$ are called \emph{end vertices} of $P$. 
If $P$ has length one together with any two vertices $u, v \in V(P)$, then we call $P$ a linear $u$--$v$ path.
    
For a linear $u$--$v$ path $P$, all the vertices other than $u$ and $v$ are called \emph{internal vertices} and let $\mathrm{int}(P)$ be the set of internal vertices of $P$.
Let $H$ be an $k$-graph. 
For two distinct vertices $u, v \in V(H)$, a linear $u$--$v$ path $P$ contained in $H$ is called a \emph{bare path} if for every internal vertex $w$, we have $N_H(w)=N_P(w)$.
For a hypergraph $H$, its subgraph $P$ isomorphic to a linear path is called a \emph{semi-bare path} if we have $N_H(w)=N_P(w)$ whenever $w$ is not contained in the ends.
We note that a bare path is only defined when we specify the end vertices $u$ and $v$ while a semi-bare path is defined on a linear path.

A $k$-graph $H$ is called a \emph{linear $k$-hypertree}, or simply hypertree, if for every $\{u, v\} \in \binom{V(H)}{2}$, there exists a unique linear $u$--$v$ path in $H$.
If each component of $H$ is a linear $k$-hypertree, then we say that $H$ is a \emph{linear $k$-forest}.
Recall that an edge with at least $k-1$ degree $1$ vertices is called a \emph{leaf edge}.

It is a well-known fact~\cite{Krivelevich2010} that if a ($2$-)tree has only a small number of leaves then it contains many vertex-disjoint bare paths.
This result can be generalized into hypertrees.
For example, the proof in~\cite{Seonghyuk2022} can be easily adapted to linear $k$-hypertreees as follows.

\begin{lemma}\label{lem:leaf_path_hypertree}
    Let $k, \ell, m \geq 2$ be an integer.
    Let $T$ be a linear $k$-hypertree with at most $\ell$ leaf edges.
    Then there exist edge-disjoint semi-bare paths $P_1, P_2, \cdots, P_s$ of length $m+1$ such that
    \begin{align*}
         e(T-P_1-P_2-\cdots-P_s)\leq 6m\ell + \frac{2e(T)}{m+1}
    \end{align*}
    where $T-P$ is a graph obtained by deleting all the vertices of $P$ except ends.
\end{lemma}
We note that by applying this result componentwise, the same conclusion holds when $T$ is a linear $k$-forest.

\subsection{Other tools}

As we plan to find many perfect fractional matchings to produce a pseudorandom fractional matching, we need the following basic lemma from mathematical programming.

\begin{lemma}[Farkas' Lemma (see~\cite{Gale1951})]\label{lem:farkas}
    For a real matrix $A \in \mathbb{R}^{m \times n}$ and a real vector $\mathbf{b} \in \mathbb{R}^m$, exactly one of the following is true.
    \begin{itemize}
        \item There exists $\mathbf{x} \in \mathbb{R}^n$ such that $A\mathbf{x}=\mathbf{b}$ and $\mathbf{x}_i \geq 0$ for all $1 \leq i \leq n$.
        \item There exists $\mathbf{y} \in \mathbb{R}^m$ such that $(A^T\mathbf{y})_i \geq 0$ for all $1 \leq i \leq n$ and $\mathbf{b}^T\mathbf{y}<0$.
    \end{itemize}
\end{lemma}

We frequently use the following well-known concentration inequality (see, for example, ~\cite{alonprobabilistic})

\begin{lemma}[The Chernoff Bound]\label{lem:chernoff}
Let $X_1, X_2, \cdots, X_n$ be i.i.d. Bernoulli random variables and let $X= \sum_{i=1}^n X_i$. Then for $\ve \in (0, 1)$, $$\mathbb{P}(|X-\mathbb{E}[X]| \geq \ve \mathbb{E}[X] ) \leq 2\exp\left(-\frac{\ve^2 \mathbb{E}[X]}{3}\right).$$     
\end{lemma}


\section{Regularization}\label{sec:regularization}

In order to obtain a pseudorandom hypergraph in the given dense linear hypergraph, we first prove that if a linear hypergraph $H$ has large minimum degree, then $H$ contains a perfect fractional matching.

\begin{theorem}\label{thm:fractional_matching}
    If $H$ is an $n$-vertex linear $k$-graph with $\delta(H) \geq \frac{n}{k}$, 
    then $H$ has a perfect fractional matching. 
\end{theorem}

\begin{proof}[Proof of \Cref{thm:fractional_matching}] 
    Let $A$ be the $|V(H)| \times |E(H)|$ incidence matrix where the rows are indexed by vertices of $H$ and the columns are indexed by the edges of $H$ with $A_{ve} = 1$ if $v \in e$ and $0$ otherwise.
    Let $\mathbf{b} = \mathbf{1}$ be the size $|E(H)|$ column vector with all entries $1$.
    Then $\psi$ is a perfect fractional matching if and only if $A\psi = \mathbf{b}$ when we consider $\psi$ as a vector in $\mathbb{R}^{E(H)}$.
    Suppose such $\psi$ does not exist.
    Then by \Cref{lem:farkas}, there exists $\mathbf{y} \in \mathbb{R}^{V(H)}$ such that $(A^T\mathbf{y})_e \geq 0$ for all $e \in E(H)$ and $\mathbf{b}^T\mathbf{y}<0.$
    In other words,
    \begin{align}
        & \,\,\,\sum_{v \in e} \mathbf{y}_v \geq 0 \text{ for every } e \in E(H), \label{eq:farkas-1}\\
        & \sum_{v \in V(H)} \mathbf{y}_v <0.\label{eq:farkas-2}
    \end{align}

    Let $X$ be the set of vertices with $\mathbf{y}_v \geq 0$, $Y = V(H) \setminus X$, and $u \in Y$ be the vertex with the smallest $\mathbf{y}_u$ (i.e., $|\mathbf{y}_u|$ is the largest among the vertices in $Y$). 
    
    We note that $\sum_{v \in V(H)} \mathbf{y}_v <0$ implies that $Y$ is nonempty so one can choose such $u$.
    Let $E_u$ be the set of edges containing $u$. For each $0 \leq i \leq k-1$, let $x_i$ be the number of edges $e \in E_u$ such that $|e \cap X|=i$.
    From \eqref{eq:farkas-1}, we have $x_0=0$. 
    As $G$ is a linear hypergraph, for every two distinct edges $e_1, e_2 \in E_u$, we have $e_1 \cap e_2 = \{u\}$.
    Thus, by considering the size of the neighborhood of $u$ in $X$, we have 
    \begin{align}\label{eq:size_of_nbhd}
    x_1 + 2x_2 + \cdots + (k-1)x_{k-1} \leq |X|.
    \end{align}

    We now consider the difference $(\sum_{v \in V(H)} \mathbf{y}_v) - (\sum_{e \in E_u} \sum_{v \in e} \mathbf{y}_v)$. 
    As the first term is negative and the second term is nonnegative, this difference is negative.
    If a vertex $v$ is contained in the neighborhood of $u$, then $\mathbf{y}_v$ is counted exactly once in the both of the $\sum_{v \in V(H)} \mathbf{y}_v$ and $\sum_{e \in E_u} \sum_{v \in e} \mathbf{y}_v$. Therefore,

    \begin{align*}
        (\sum_{v \in V(H)} \mathbf{y}_v) - (\sum_{e \in E_u} \sum_{v \in e} \mathbf{y}_v) & = \sum_{v \in V(H) \setminus N_H(u)} \mathbf{y}_v - d_H(u)\mathbf{y}_u \\
        & \geq \sum_{v \in (V(H) \setminus N_H(u)) \cap Y} \mathbf{y}_v - d_H(u)\mathbf{y}_u \\
        & \geq  (|(V(H) \setminus N_H(u)) \cap Y| - d_H(u)) \mathbf{y}_u.
    \end{align*}

    The final inequality is from the minimality of $\mathbf{y}_u.$
    Thus, we have $|(V(H) \setminus N_H(u)) \cap Y| - d_H(u) > 0$, as otherwise we get a contradiction from $(|(V(H) \setminus N_H(u)) \cap Y| - d_H(u)) \mathbf{y}_u < 0.$ 
    As $|N_H(u) \cap Y| = \sum_{i = 1}^{k-2} (k-1-i)x_i$, we have 

    \begin{align}\label{eq:number_of_covered_vertices}
        |(V(H) \setminus N_H(u)) \cap Y| - d_H(u) & = |Y| - \sum_{i = 1}^{k-1} (k-1-i)x_i - \sum_{i = 1}^{k-1}x_i \nonumber \\
        &= |Y| -  \sum_{i = 1}^{k-1} (k - i) x_i >0.
    \end{align}
    
    By combining \eqref{eq:size_of_nbhd} and \eqref{eq:number_of_covered_vertices}, we have $k \sum_{i = 1}^{k-1} x_i < |X|+|Y| = n$ and therefore $d_H(u) = \sum_{i = 1}^{k-1} x_i < \frac{n}{k}$.
    It contradicts the minimum degree condition so such $\mathbf{y}$ does not exist. 
    Therefore, we have a perfect fractional matching of $H$. 
\end{proof}

By combining \Cref{thm:fractional_matching} and \Cref{lem:make-pseudorandom}, we obtain the following corollary, which is the main lemma for proving \Cref{thm:tree-tiling}.

\begin{corollary}\label{cor:regularization}
    Let $0<\frac{1}{n}, \frac{1}{C} \ll \lambda < 1.$
    Let $H$ be an $n$-vertex linear $k$-graph with $\delta(H) \geq \frac{n}{k} + C$, then there exists a $D$-regular multi-hypergraph $F$ on the same vertex set for some positive integer $D$ such that 
    \begin{enumerate}
        \item $F_{simp} \subseteq H$ and 
        \item for all $e\in E(F)$, we have $\mu_{F}(e) \leq \lambda D$.
    \end{enumerate}
\end{corollary}

\begin{proof}[Proof of \Cref{cor:regularization}]
    We claim that $H$ contains a $\lambda$-pseudorandom fractional matching $\varphi$.
    If the claim is true, then we may assume that all the values of $\varphi$ are rational so there exists an integer $D$ such that $D \cdot \varphi(e)$ is an integer for every $e \in E(H)$.
    Then we construct $F$ by adding an edge $e \in E(H)$ by $D \cdot \varphi(e)$ times.
    Then as $\sum_{e \ni v} \varphi(e)=1$ for every $v \in V(H)$, the hypergraph $F$ is $D$-regular and each edge $e \in E(H)$ is contained in $D \cdot \varphi(e) \leq \lambda D$ times in $F$. 
    Therefore, $F$ is the desired multi-hypergraph.
    We now prove the existence of $\varphi$.
    We apply \Cref{lem:make-pseudorandom} to $H$ with the following parameters. 
    
    \begin{center}
        \begin{tabular}{ |c|c|c|c|c|c| } 
         \hline
         parameter & $C$ & $n/k$ & $k/n$ & $\lambda$ & $\frac{\lambda}{4}$\\ 
         \hline
         plays the role of & $C$ & $D$ & $\delta$ & $\eta$ & $\gamma$ \\ 
         \hline
        \end{tabular}
    \end{center}
    
    By \Cref{thm:fractional_matching}, for every subhypergraph $F \subseteq H$ with $\Delta(F) \leq C$, $H \setminus F$ contains a perfect fractional matching.
    We also have that $\delta D (\gamma + 2^{-\gamma C}) = \frac{\lambda}{4} + 2^{-\frac{\lambda}{4}C} \leq \frac{\lambda}{2}$ with sufficiently large constant $C$.
    Therefore, by \Cref{lem:make-pseudorandom}, we have a $\lambda$-pseudorandom fractional matching $\varphi$. 
\end{proof}


\section{Proof of \Cref{thm:tree-tiling}}\label{sec:tree-tiling}

\begin{proof}[Proof of \Cref{thm:tree-tiling}]
    Let $m$ be the number of edges of $T$ and $t=m(k-1)+1$ be the number of vertices of $T$.
    Choose $\eta, \mu, C$ satisfying $$0<\frac{1}{n} , \frac{1}{C} \ll \eta \ll \mu \ll \frac{1}{k}, \frac{1}{t}, \varepsilon.$$
    Let $H$ be an $n$-vertex linear $k$-graph with $\delta(H) \geq \frac{n}{k} + C$.
    Then by \Cref{cor:regularization}, there exists a $D$-regular multi-hypergraph $H'$ on the same vertex set such that $H'_{simp} \subseteq H$ and $H'$ is $(D, 0, \eta)$-pseudorandom for some integer $D \geq 1$.
    We now construct an auxiliary multi-$t$-graph $\mathcal{H}$ as follows.
    We have $V(\mathcal{H}) = V(H')$ and for each labeled copy of $T$ in $H'$, we add an edge to $\mathcal{H}$.
    We note that we consider two copies of $T$ in $H'$ to be different if the labelings of the vertices are different even if the sets of vertices are the same. 
    Similarly, if the sets of edges are different, then two copies of $T$ are considered different even if the labelings of the vertices are the same.
    We now claim that $\mathcal{H}$ is $(D', \mu, \mu)$-pseudorandom for some integer $D' \geq 1$.
    If the claim is true, then by \Cref{thm:pippenger_spencer}, $\mathcal{H}$ has a matching that covers all but at most $\varepsilon n$ vertices.
    By the construction, it gives the vertex-disjoint copies of $T$ that cover all but at most $\varepsilon n$ vertices of $H$. It concludes the proof.

    We now prove the claim.
    We first prove that there exists an integer $D'$ such that for every $u \in V(\mathcal{H})$, we have $d_{\mathcal{H}}(u) = D'(1 \pm \mu)$.
    We choose an arbitrary vertex $v \in V(T)$ and count the number of copies of $T$ in which $v$ is mapped to $u$.
    We order the edges of $T$ by starting a breadth first search(BFS) from $v$. 
    That is, we construct an ordering by placing the edges incident to $v$ first, and for each edge $e$ that is placed, we add the edges incident to $e$ that is not placed yet at the end of the ordering.
    Let $e_1, e_2, \ldots, e_m$ be the edges of $T$ in the ordering.
    We greedily count the number of copies of $T$ in which $v$ is mapped to $u$ by extending the embedding one by one.
    For each $0 \leq j \leq m-1$, once the image of $e_1, \ldots, e_j$ and the labeling of the vertices in the image are fixed, let $u'$ be the image of the vertex in $(\bigcup_{i=1}^j e_i) \cap e_{j+1}$.
    Then the number of ways to embed $e_{j+1}$ is at most $d_{H'}(u') (k-1)! = D (k-1)!$ where $(k-1)!$ is the number of ways to label the vertices in the image $e_{j+1}$.
    Similarly, the number of ways to embed $e_{j+1}$ is at least $(D - j(k-1) \eta D)(k-1)! \geq (1-t\eta)D(k-1)!$ as $j(k-1) \eta D$ is the maximum possible number of edges that incident to $u'$ and contains another vertex that is already used in the embedding.
    Therefore, the number of ways to embed $T$ in which $v$ is mapped to $u$ is at most $D^{m}(k-1)!^{m}$ and at least $(1-t\eta)^{m}D^{m}(k-1)!^{m} \geq (1-\mu)D^{m}(k-1)!^{m}$.
    As it is independent from the choice of $u$ and $v$, by taking $D' = tD^{m}(k-1)!^{m}$, every vertex in $\mathcal{H}$ has degree $D'(1 \pm \mu)$.

    We finally show that the codegree of $\mathcal{H}$ is at most $\mu D'$.
    We fix two distinct vertices $u_1, u_2 \in V(\mathcal{H})$. 
    We choose arbitrarily two distinct vertices $v_1, v_2 \in V(T)$ and aim to count the number of $T$ copies in $H'$ in which $v_1$ is mapped to $u_1$ and $v_2$ is mapped to $u_2$.
    We similarly order the edges of $T$ by using BFS starting from $v_1$ and let $e_1, e_2, \ldots, e_m$ be the resulting ordering.
    Let $j \in [m]$ be the smallest index such that $v_2$ is contained in $e_j$.
    Then we first embed the edges $e_1, \ldots, e_{j-1}$ and the number of ways to do so is at most $D^{j-1}(k-1)!^{j-1}$.
    After we embed $e_1, \ldots, e_{j-1}$ and choose the labeling of the vertices, the number of ways to embed $e_j$ is at most $\eta D (k-2)! \leq \eta D (k-1)!$ as the codegree of the image of the vertex in $(\bigcup_{i=1}^j e_i) \cap e_{j+1}$ and $u_2$ is at most $\eta D$.
    The number of ways to embed $e_{j+1}, \ldots, e_m$ is at most $D^{m-j}(k-1)!^{m-j}$ by the same argument.
    Therefore, the number of ways to embed $T$ in which $v_1$ is mapped to $u_1$ and $v_2$ is mapped to $u_2$ is at most $\eta D^{m}(k-1)!^{m}$.
    As it is independent of the choice of $v_1, v_2$, the codegree of $u_1$ and $u_2$ is at most $t^2 \eta D^{m}(k-1)!^{m} \leq \mu D'$, which concludes the proof of the claim.
\end{proof}


\section{Proof of~\Cref{thm:min_deg_almost_subdivision}}\label{sec:subdivison}

\subsection{Connecting lemma}\label{sec:connection-lemmas}
We first show that the minimum degree condition implies that there are many paths between any two vertices. For the later discussion, we actually show that the same conclusion holds with a weaker minimum degree condition.

\begin{lemma}\label{lem:resorvoir}
    Let $\varepsilon>0$ be a real number, $k \geq 2$, and $n \geq 2k\varepsilon^{-1}$ be an integer.
    Let $H$ be an $n$-vertex linear $k$-graph with $\delta(H) \geq (\frac{1}{2(k-1)}+\varepsilon)n$.
    Then for every two distinct vertices $u, v \in V(H)$, there are at least $\varepsilon n/2k$ $u$--$v$ paths of length $2$ with their internal vertices pairwise disjoint.
\end{lemma}

\begin{proof}
    We fix two distinct vertices $u, v \in V(H)$ and choose the collection of $u$--$v$ paths greedily as follows.
    Let $P_1, \ldots, P_t$ be the collection of $u$--$v$ paths that we have chosen and assume that $t < \varepsilon n/2k$.
    Let $X$ be the set of internal vertices of $P_1, \ldots, P_t$.
    Then $|X| \leq t(2k-3) <\varepsilon n-1$.
    Thus, we have $$d_H(u; V(H) \setminus X) > (\frac{1}{2(k-1)}+\varepsilon)n - \varepsilon n +1 = \frac{n}{2(k-1)} +1 . $$
    By the same reason, $d_{H}(v; V(H) \setminus X) > \frac{n}{2(k-1)}+1$.
    As there are at most one edge containing $\{u, v\}$, there are more than $\frac{n}{2(k-1)}$ edges that are incident to $u$ ($v$, respectively) which does not intersects $\{v\} \cup X$ ($\{u\} \cup X$, respectively). 
    Therefore, there exist two distinct edges $e_1, e_2$ in $H \setminus X$ such that $u \in e_1, v \in e_2$ and $e_1 \cap e_2$ contains a unique vertex which is not equal to $u$ and $v$.
    Then this $e_1$ and $e_2$ form a $u$--$v$ path of length $2$ where all the vertices other than $u$ and $v$ are disjoint from $X$.
\end{proof}


\subsection{Proof of \Cref{thm:min_deg_almost_subdivision}}

We now prove \Cref{thm:min_deg_almost_subdivision}. 
Indeed, we prove the following more general statement that covers both of the possibilities of $F$.

\begin{theorem}\label{thm:min_deg_embedding_general}
    Let $k \geq 2$ be an integer and $\ve, \eta>0$.
    Then there exists $n_0=n_0(k, \ve, \eta)$ and $\mu = \mu(k, \ve, \eta)>0$ such that the following is true.
    
    Let $H$ be a linear $k$-graph on $n$ vertices where $n \geq n_0$ and $\delta(H) \geq (\frac{1}{k} + \ve) n$.
    Let $F$ be a hypergraph that satisfies the following.
    \begin{enumerate}
        \item $|V(F)| \leq (1-\eta) n$.
        \item There exists a sequence of $u_i$--$v_i$ paths $P_1, \ldots, P_t$ for some $t \leq \mu n$ such that $P_i$ is a bare path of length $2$ in $F - \bigcup_{j=1}^{i-1} (\mathrm{int}(P_j))$.
        \item $F-\bigcup_{j=1}^{t} (\mathrm{int}(P_j))$ is a hyperforest with at most $\mu n$ leaf edges.
    \end{enumerate} 
    Then $H$ contains $F$ as a subgraph.
\end{theorem}
Then one can deduce \Cref{thm:min_deg_almost_subdivision} from \Cref{thm:min_deg_embedding_general}.
\begin{proof}[Proof of \Cref{thm:min_deg_almost_subdivision}]
    When $F$ is a hypertree with at most $\mu n$ leaf edges, then by taking $t=0$, one can embed $F$ into $H$.
    When $F$ is a subdivision of a graph $G$, then we choose $P_i$ to be a path of length $2$ contained in each path corresponding to an edge of $G$.
    Then $t \leq {|V(G)| \choose 2} \leq \mu n$ and all of the $P_i$ are bare paths.
    Let $T$ be the graph obtained by deleting all the internal vertices of $P_i$ from $F$.
    Then it suffices to show that $T$ is a forest with at most $\mu n$ leaf edges.
    This can be proved from the following fact:
    Each component of $T$ contains exactly one center vertex $v$ of $F$ and the union of at most $\mu \sqrt{n}$ linear paths whose intersection is exactly $\{v\}$.
    Thus, $T$ is a hyperforest and each component of $T$ has at most $\mu \sqrt{n}$ leaf edges so $T$ has at most $\mu \sqrt{n} \times |V(G)| \leq \mu n$ leaf edges.
\end{proof}

Before proving \Cref{thm:min_deg_embedding_general}, we note the following lemma which can be proved by a simple greedy algorithm.
\begin{lemma}\label{lem:greedy}
    Let $H$ be a linear $k$-graph with $\delta(G) \geq m$ and $T$ be a linear $k$-forest with at most $m/k$ vertices.
    Then $H$ contains $T$ as a sub-hypergraph.
\end{lemma}

We now prove \Cref{thm:min_deg_embedding_general}.
\begin{proof}[Proof of Theorem~\ref{thm:min_deg_embedding_general}]

    We take parameters so that $$0<\frac{1}{n} \ll \mu \ll \delta \ll \delta' \ll \ve , \eta, \frac{1}{k}.$$
    Let $F$ be a graph which satisfies the conditions of \Cref{thm:min_deg_embedding_general} with $P_i$s and $T = F -\bigcup_{j=1}^{t} (\mathrm{int}(P_j))$.
    Let $m = \lceil 10k/\delta \rceil$ and $M$ be the smallest integer which is at least $1/\sqrt{\mu}$ and $M-2$ is divisible by $m+2$.
    By \Cref{lem:leaf_path_hypertree}, $T$ has edge-disjoint semi-bare paths $P_1', \ldots, P_s'$ of length $M+2$ such that 
    $$|E(T-P_1'-\cdots-P_s')| \leq 6(M+2) \cdot \mu n + \frac{2e(T)}{M+3} \leq 10 \sqrt{\mu} n.$$
    We note that $s M \leq n$ so $s \leq \sqrt{\mu} n$.
    Let $s^1_i, s^2_i \in V(P_i')$ be the vertices that are contained in the first and last edge of $P_i'$ and have degree $2$ in $P_i'$.
    Let $Q_i$ be the $s^1_i$--$s^2_i$ path obtained by deleting the first and last edge of $P_i'$, of length $M$.
    Then by choice, $Q_i$ is a bare path of $T$ for each $i \in [s]$ and they are pairwise vertex-disjoint.
    Let $T'$ be the graph obtained by deleting all the internal vertices of $Q_1, \ldots, Q_s$.
    We note that 
    $$|E(T')| \leq |E(T-P_1'-\cdots-P_s')| + 2s \leq 12\sqrt{\mu} n.$$
    We now aim to embed $T'$ and extend this embedding to an embedding of $T$ and $F$. 
    To achieve this, we claim that $H$ contains an appropriate reservoir.
    \begin{claim}\label{clm:resorvior}
        There is a subset $R \subseteq V(H)$ that satisfies the followings with $H' = H\setminus R$ and $n' = |V(H')|$.
        \begin{enumerate}[(i)]
            \item $n' \geq (1-\delta')n$.\label{cond:min_deg}
            \item $\delta(H') \geq (\frac{1}{k}+\frac{\ve}{2})n'$. 
            \item For every pair of distinct vertices $u, v$, there are at least $\delta n$ internally disjoint paths of length two connecting $u$ and $v$ with all vertices other than $u$ and $v$ are in $R$.\label{cond:resorvoir}
        \end{enumerate}
    \end{claim}

    \begin{claimproof}[Proof of \Cref{clm:resorvior}]
        We choose $R\subseteq V(H)$ to be the random subset where each vertex is chosen with probability $p=\delta'/2$ independently, and we will prove that with positive probability, this random set $R$ is the desired subset.

         By Chernoff bound (\Cref{lem:chernoff}), $|R| \geq \frac{3}{4}\delta' n$ with probability at most 
        $$ 2\exp{\left(-\frac{\delta' n}{12}\right)} = o(1) $$
        so the first condition holds with a probability of at least $3/4$.
        Similarly, for each vertex $v \in V(H)$, we have 
        $$\mathbb{P}\left(d_{H'}(v) < (1-p)^{k-1} \left(\left(\frac{1}{k}+\ve\right)n \cdot (1-\delta')\right) \leq 2\exp{\left(-\delta'^2 (1-p)^{k-1}\left(\frac{1}{k}+\ve\right)\frac{n}{3}\right)}\right) = o\left(\frac{1}{n}\right).$$
        Thus, by the union bound and the fact that $(1-p)^{k-1} (\frac{1}{k}+\ve)n \cdot (1-\delta') \leq (\frac{1}{k}+\frac{\ve}{2})n'$, the second condition holds with probability at least $3/4$.
        Finally, for every pair of distinct vertices $u, v$, \Cref{lem:resorvoir} implies that there are at least $\delta' n$ paths connecting $u$ and $v$ with pairwise disjoint internal vertices.
        The events that internal vertices of each of those paths are in $R$ are mutually independent.
        Thus, the probability that there are less than $p^{k-1} \delta' n/2 \leq \delta n$ paths with all internal vertices in $R$ is at most
        $$ 2\exp(-p^{k-1} \delta' n/12) = o(1/n^2)$$
        by \Cref{lem:chernoff}.
        Thus, by the union bound, the third condition holds with probability at least $3/4$.
        Therefore, $R$ satisfies \ref{cond:min_deg}--\ref{cond:resorvoir} with probability at least $1/4$.
    \end{claimproof}

    We now fix a choice of $R$ that satisfies \Cref{clm:resorvior}.
    As $T'$ has at most $20\sqrt{\mu} n$ edges, it has at most $20k \sqrt{\mu} n$ non-isolated vertices.
    By \Cref{lem:greedy}, one can greedily embed $T'$ into $H'$ except isolated vertices.
    Let $X \subseteq V(H')$ be the image of $V(T')$ of this embedding and let $H'' = H' \setminus X$ and $n'' = |V(H'')|$.
    As $20k\sqrt{\mu} n \leq \ve n/10$, we have $\delta(H'') \geq (\frac{1}{k}+\frac{\ve}{4})n''$.
    Then by applying \Cref{thm:tree-tiling} to $H''$, we conclude that $H''$ has vertex disjoint paths of length $m$ that cover all but at most $\delta n''$ vertices.
    
    We partition the set of paths of length $m$ into sets $\mathcal{S}_1, \mathcal{S}_2, \ldots, \mathcal{S}_s$ of $(M-2)/(m+2)$ paths (if there is some leftover, we discard them).
    We note that $|V(T)| \leq (1-\eta) n$ and $Q_i$s are pairwise vertex-disjoint so $s \times (M(k-1)+1) \leq (1-\eta) n$. 
    Thus, $s \times \frac{M-2}{m+2} \leq \frac{(1-\eta)n}{(m+2)(k-1)}$.
    As the number of paths of length $m$ we obtained is at least $\frac{(1-\delta')n''}{m(k-1)+1} \geq \frac{(1-\eta)n}{(m+2)(k-1)}$, one can find $\mathcal{S}_1, \mathcal{S}_2, \ldots, \mathcal{S}_s$.
    Let $\mathcal{S}_j = \{S^j_1, \ldots, S^j_{(M-2)/(m+2)}\}$ and assume that each $S^j_i$ is an ${u^j_i}$--${v^j_i}$ path.

    Then by \ref{cond:resorvoir}, there exists at least $\delta n$ paths of length $2$ connecting ${v^j_{i+1}}$ and ${u^j_i}$ such that all internal vertices are in $R$ for each $i \in [(M-2)/(m+2)-1]$.
    We choose one of those paths for every $j \in [s]$ and $i \in [(M-2)/(m+2)-1]$ so that they are pairwise vertex-disjoint.
    As the total number of paths is at most $n/m \leq \delta n/10 k$, the number of vertices that are used is at most $\delta n/5$ so this selection can be done greedily.
    We now connect all the paths in each $\mathcal{S}_j$ by using the previously chosen paths.
    As a result, we have $s$ paths of length $M-4$ which are pairwise vertex-disjoint.
    Let $X'$ be the set of vertices that are in these paths of length $M-4$.
    We also note that this collection of paths uses at most $\delta n/5$ vertices in $R$.

    We now embed isolated vertices of $T'$ by any vertices in $V(H) \setminus (R \cup X \cup X')$. 
    As $|V(T)| \leq (1-\eta)n < n-|R|$, this can be done greedily. 
    We now embed paths $Q_i$ into $H$ one by one by selecting one path of length $M-4$ which has not been used before and connects each end of the path and the embedded image of $s_i^1$ and $s_i^2$, respectively by using \ref{cond:resorvoir}.
    We can see that in any step, the number of vertices used in $R$ is at most $\delta n/5 + 2ks \leq \delta n/4$.
    Thus, for each step, there are at least $3\delta n/4$ paths between any two vertices such that all the internal vertices are in $R$ and do not contain previously used vertices of $R$.
    This finishes the embedding of $T$.

    To finish the embedding of $H$, we need to embed $P_1, \ldots, P_t$.
    Observe that one can greedily select pairwise disjoint length two paths between the images of $u_i$ and $v_i$ for each $i \in [t]$ such that all the internal vertices are in $R$ and disjoint from all the previous embedding.
    As the total number of used vertices of $R$ is at most $\delta n/4 + 2kt \leq \delta/2$, by \ref{cond:resorvoir}, there are at least $\delta n/2$ paths between any two vertices such that all the internal vertices are contained in $R$ and do not contain previously used vertices of $R$.
    Together with these paths, one can embed all of the $P_1, \ldots, P_t$. 
    This gives an embedding of $H$ which concludes the proof.
\end{proof}


\section{Concluding remarks}\label{sec:concluding_remarks}

We finally note that with a slightly stronger minimum degree condition, one can prove a robust version of \Cref{thm:min_deg_almost_matching} by following the simplified idea of~\cite{Kelly2024}.
Let $\ve, \ve'>0$ be a given constants and $n$ be a sufficiently large integer.
If a linear hypergraph $G$ satisfies $\delta(H) \geq (\frac{1}{k}+\ve)n$, then \Cref{lem:make-pseudorandom} implies the existence of a $(D, 0, \frac{\log^2 n}{n})$-pseudorandom multi-hypergraph $F$ whose simplification is contained in $H$.
Then for $p =  \frac{(\log n)^c}{n^{1/(k-1)}}$, a $p$-random subset $X \subseteq V(F)$, with high probability satisfies that $F[X]$ is an almost regular graph with small maximum codegree condition. 
Thus, by partitioning $V(H)$ into $V_1, \ldots, V_k$ with $k=1/p$ by selecting each element uniformly at random, each $H[V_i]$ has an almost perfect matching.
By taking their union, we obtain an almost perfect matching in $H$. 
Furthermore, for any edge $e \in E(H)$, the probability that $e$ is contained in the matching is bounded by $\mathbb{P}(e \subseteq V_i \text{ for some $i$}) = p^{k-1}$.  
We can apply the same argument for any set of pairwise disjoint edges $e_1, \ldots, e_t$ so this gives a $\frac{(\log n)^{c'}}{n}$-spread probability measure on a set of matchings with $(1-\ve')n/k$ edges for some constant $c'>0$.
As a corollary, there are $\exp(\frac{1-\varepsilon'}{k}n - O(n \log \log n))$ matchings of size $(1-\varepsilon')n/k$.
Also, together with a theorem by Frankston, Kahn, Narayanan, and Park~\cite{Frankston2021}, we conclude the following.
\begin{corollary}
    There exists a constant $c>0$ such that the following holds.
    Let $H$ be an $n$-vertex $k$-graph.
    If $p \geq \frac{(\log n)^c}{n}$ and $\delta(H) \geq (\frac{1}{k}+\varepsilon)n$, then $p$-random subgraph of $G$ contains a matching of size $(1-\varepsilon)n/k$ with high probability.
\end{corollary}
We conjecture that the power $c$ of the $\log n$ can be improved to $1$.

Another question we want to note is whether the minimum degree condition of \Cref{thm:min_deg_almost_subdivision} can be improved for specific hypergraphs. 
In particular, it is interesting to decide the minimum degree threshold for the existence of an almost spanning linear cycle in a linear hypergraph. As we saw in \Cref{prop:lower-bound}, the minimum degree condition is tight for linear hypergraphs containing a near-perfect matching. However, for other linear hypergraphs, it is not clear what would be the correct threshold.

We note that our proof yields the following two relevant facts.
First, if a linear hypergraph $H$ is regular, then we can find an almost spanning linear cycle with a weaker condition $\delta(H) \geq (\frac{1}{2k-2} + o(1))n$.
This can be proved by applying the proof of \Cref{thm:min_deg_embedding_general} to $H$ directly without applying \Cref{cor:regularization}.
This bound is asymptotically best possible by considering the case when $H$ is a disjoint union of two (almost-) Steiner systems of almost equal order.
Second, a similar technique stated in the above paragraph proves that there are at least $\exp(\frac{(1 - o(1))n \log n}{k-1}- o(n \log n))$ linear cycles with $(1 - o(1))n$ vertices in $H$ when $\delta(H) \geq (\frac{1}{k} + o(1))n$.

\subsection*{Acknowledgement}
The authors would like to thank their advisor Jaehoon Kim for his helpful advice to improve the presentation of this paper.

SI and HL are supported by the National Research Foundation of Korea (NRF) grant funded by the Korea government(MSIT) No. RS-2023-00210430 and by the Institute for Basic Science (IBS-R029-C4).

\end{document}